\documentclass[11pt]{amsart}
\usepackage{amsfonts}
\usepackage{mathrsfs}
\usepackage{amsmath}
\usepackage{amsmath, amsthm, amssymb}
\usepackage{color}

\newcommand{\T}{\theta}

\renewcommand{\(}{\left( }
\renewcommand{\)}{\right) }

\renewcommand{\theequation}{\theequation. \arabic{equation}}
\numberwithin{equation}{section}
\newtheorem{thm}{Theorem}[section]
\newtheorem{cor}[thm]{Corollary}

\newtheorem{rem}[thm]{Remark}
\newtheorem{prop}[thm]{Proposition}
\newtheorem{defn}[thm]{Definition}

\def\squarebox#1{\hbox to #1{\hfill\vbox to #1{\vfill}}}

\begin{document}
\title[On the Hahn polynomials]
{ A $q$-extension of a partial differential equation and the Hahn polynomials}
\author{Zhi-Guo Liu}
\date{\today}
\address{East China Normal University, Department of Mathematics, 500 Dongchuan Road,
Shanghai 200241, P. R. China} \email{zgliu@math.ecnu.edu.cn;
liuzg@hotmail.com}
\thanks{The  author was supported in part by
the National Science Foundation of China}
\begin{abstract}
Using the theory of analytic functions of several complex variables, we prove that if an analytic
function in several variables satisfies a system of $q$-partial differential equations,
then,  it can be expanded in terms of the product of the homogeneous Hahn polynomials.
Some non-trivial applications of this expansion theorem to $q$-series are discussed.
\end{abstract}
\maketitle
\section{Introduction}
If $f(x)$ is an analytic function at $x=0$ and $\alpha$ a nonzero complex number, then, $f(\alpha x+y)$ is
a two-variables analytic function of $x$ and $y$ in  a neighbourhood of $(0, 0) \in \mathbb{C}^2.$
There is no obvious way to find the $q$-extension of $f(\alpha x+y)$, but a direct computation shows that
$f(\alpha x+y)$ satisfies the partial differential equation $f_x=\alpha f_y.$

Conversely, in the proposition below, we prove that all the solutions of the partial differential equation $f_x=\alpha f_y$ have the form $f(\alpha x+y).$
Thus, finding the $q$-extension of the function $f(\alpha x+y)$ is equivalent to
finding the $q$-extension of the partial differential equation $f_x=\alpha f_y$.

The motivation of this paper is to find the $q$-extension of $f(\alpha x+y)$ through a $q$-partial differential equation.

In this paper, we use $\mathbb{C}$ to denote the set of all complex numbers and
$\mathbb{C}^k=\mathbb{C}\times \cdots \times \mathbb{C}$ the set of all $k$-dimensional
complex numbers.
\begin{prop} \label{ppmliu}
If $\alpha$ is a nonzero complex number and $f(x, y)$ is a two-variables analytic function in a neighbourhood
of $(0, 0) \in \mathbb{C}^2$, satisfying the partial differential equation
$f_x(x,y)=\alpha f_y(x, y)$ with $f(0, y)=f(y)$, then,  we have $f(x, y)=f(\alpha x+y).$
\end{prop}
\begin{proof}
We use the method of power series  to solve the partial differential equation in the proposition.
Since $f$ is analytic at $(0, 0) \in \mathbb{C}^2$£¬ $f$ is an analytic function of $x$ at $x=0$.
Thus, we may assume that  near $(x, y)=(0, 0),$
\[
f(x, y)=\sum_{n=0}^\infty A_n(y)x^n.
\]
If this is substituted into $ f_x(x,y)=f_y(x, y)$, we immediately conclude that
\[
\sum_{n=1}^\infty n A_n(y)x^{n-1}= \alpha \sum_{n=0}^\infty A_{n}'(y)x^n.
\]
Equating the  coefficients of $y^n$ on both sides of the above equation,
we find that for each integer $n\ge 1, A_n(y)=\alpha {A_{n-1}'(y)}/{n}.$
By iteration, we deduce that $A_n(y)={{\alpha}^n A_0^{(n)}(y)}/{n!}.$  It is obvious that
$A_0(y)=f(0, y)=f(y).$ Thus, using the Taylor expansion, we immediately deduce that
$f(x, y)=f(\alpha x+y),$ which completes the proof of Proposition~\ref{ppmliu}.
\end{proof}

Now we introduce some $q$-notations. For $|q|<1$ and $a \in \mathbb{C}$, we define the $q$-shifted factorials as
\begin{equation*}
(a; q)_0=1,\quad (a; q)_n=\prod_{k=0}^{n-1}(1-aq^k), \quad (a;
q)_\infty=\prod_{k=0}^\infty (1-aq^k);
\end{equation*}
and for convenience, we also adopt the following compact notation for the multiple
$q$-shifted factorial:
\begin{equation*}
(a_1, a_2,...,a_m;q)_n=(a_1;q)_n(a_2;q)_n ... (a_m;q)_n,
\end{equation*}
where $n$ is an integer or $\infty$.

\begin{defn}\label{qddfn}
For any function $f(x)$ of one variable, the  $q$-derivative of $f(x)$
with respect to $x,$ is defined as
\begin{equation*}
\mathcal{D}_{q,x}\{f(x)\}=\frac{f(x)-f(qx)}{x},
\end{equation*}
and we further define  $\mathcal{D}_{q,x}^0  \{f\}=f,$ and for $n\ge 1$, $\mathcal{D}_{q, x}^n \{f\}=\mathcal{D}_{q, x}\{\mathcal{D}_{q, x}^{n-1}\{f\}\}.$
\end{defn}
\begin{defn}\label{qpde}
A $q$-partial derivative of a function of several variables is its $q$-derivative with respect to one of those variables, regarding other variables as constants. The $q$-partial derivative of a function $f$ with respect to the variable $x$ is denoted by $\partial_{q, x}\{f\}$.
\end{defn}

The Gaussian binomial coefficients also called the $q$-binomial coefficients are
the $q$-analogs of the binomial coefficients, which
are given by
\begin{equation*}
{n\brack k}_q=\frac{(q; q)_n}{(q; q)_k(q; q)_{n-k}}.
\end{equation*}
Now we introduce the definition of the Hahn polynomials $\Phi^{(\alpha)}_n(x|q)$
which were first studied by Hahn \cite{Hahn1, Hahn2},
and then by Al-Salam and Carlitz \cite{SalamCarlitz}.
\begin{defn} \label{ascpolydefn} The Hahn polynomials are defined by
\[
\Phi^{(\alpha)}_n(x|q)=\sum_{k=0}^n {n\brack k}_q (\alpha; q)_k x^k .
\]
\end{defn}
 For our purposes, in this paper we need the homogeneous Hahn polynomials, which can
be obtained from the Hahn polynomials easily.
\begin{defn} \label{bscpolydefn} The homogeneous Hahn polynomials are defined by
\[
\Phi^{(\alpha)}_n(x, y|q)=y^n\Phi^{(\alpha)}_n(x/y|q) =\sum_{k=0}^n {n\brack k}_q (\alpha; q)_k x^k y^{n-k}.
\]
\end{defn}

It is obvious that $\Phi^{(\alpha)}_n(x, 1|q)=\Phi^{(\alpha)}_n(x|q)$. If we set $\alpha=0$ in the homogeneous
Hahn polynomials, we can obtain the homogeneous Rogers-Szeg\H{o} polynomials $h_n(x, y|q)$
(see, for example \cite{LiuRamanujanP})
\[
h_n(x, y|q)=\sum_{k=0}^n {n\brack k}_q  x^k y^{n-k}.
\]
An important difference between the Hahn polynomials and
the homogeneous Hahn polynomials is that the later
satisfy the following $q$-partial differential equation,
which does not appear in the literature.
\begin{prop}\label{liupdepp} The homogeneous Hahn polynomials satisfy the
$q$-partial differential equation
\[
\partial_{q, x}\left\{\Phi^{(\alpha)}_n(x, y|q)\right\}=\partial_{q, y}\left\{\Phi^{(\alpha)}_n(x, y|q)-\alpha \Phi^{(\alpha)}_n(qx, y|q)\right\}.
\]
\end{prop}
\begin{proof} Using the identity, $\partial_{q, x} \{ x^k\}=(1-q^k) x^{k-1},$  we immediately find that
\[
\partial_{q, x}\left\{\Phi^{(\alpha)}_n(x, y|q)\right\}
=\sum_{k=1}^n {n\brack k}_q (\alpha; q)_k (1-q^k) x^{k-1} y^{n-k}.
\]
In the same way, using the identity, $\partial_{q, y} \{ y^{n-k}\}=(1-q^{n-k}) y^{n-k-1},$ we deduce that
\[
\partial_{q, y}\left\{\Phi^{(\alpha)}_n(x, y|q)-\alpha \Phi^{(\alpha)}_n(qx, y|q)\right\}
=\sum_{k=0}^{n-1} {n\brack k}_q (\alpha; q)_{k+1} (1-q^{n-k}) x^{k} y^{n-k-1}.
\]
If we make the variable change $k+1 \to k$ in the right-hand side of the above equation, we can find that
\[
\partial_{q, y}\left\{\Phi^{(\alpha)}_n(x, y|q)-\alpha \Phi^{(\alpha)}_n(qx, y|q)\right\}
=\sum_{k=1}^{n} {n\brack {k-1}}_q (\alpha; q)_{k} (1-q^{n-k+1}) x^{k-1} y^{n-k}.
\]
From the definition of the $q$-binomial coefficients, it is easy to verify that
\[
{n\brack k}_q (1-q^k)={n\brack {k-1}}_q  (1-q^{n-k+1}).
\]
Combining the above equations, we complete the proof of Proposition~\ref{liupdepp}.
\end{proof}
\begin{defn} \label{qshiftdefn} If $f(x_1, \ldots, x_k)$ is a $k$-variable function,
then,  the $q$-shift operator $\eta_{x_j}$ on the variable $x_j$ is defined as
\[
\eta_{x_j}\{f(x_1, \ldots, x_k)\}=f(x_1, \ldots, x_{j-1},  qx_j, x_{j+1}, \ldots, x_k)
~\text{for}~j=1, 2, \ldots,  k.
\]
\end{defn}

The principal result of this paper is the following theorem for the analytic functions in several variables.

\begin{thm}\label{principalliuhthm}
If $f(x_1,y_1, \ldots, x_k, y_k)$  is a $2k$-variable
 analytic function at $(0,0, \cdots, 0)\in \mathbb{C}^{2k}$, then,
 $f$ can be expanded in terms of
 \[
 \Phi^{(\alpha_1)}_{n_1}(x_1, y_1|q)\cdots \Phi^{(\alpha_k)}_{n_k}(x_k, y_k|q)
 \]
 if and only if  $f$ satisfies the $q$-partial differential equations
 \[
 \partial_{q, x_j}\{f\}=\partial_{q, y_j}(1-\alpha_j \eta_{x_j})\{f\}~\text{for}~j=1, 2, \ldots,  k.
 \]
\end{thm}
\begin{prop}\label{qliuextension}
The case $k=1$ of Theorem~\ref{principalliuhthm} is a $q$-extension of Proposition~\ref{ppmliu}.
\end{prop}
\begin{proof}
When $k=1$, the $k$ partial differential equations in Theorem~\ref{principalliuhthm} reduce to one  partial differential equation, $\partial_{q, x_1}\{f\}=\partial_{q, y_1}(1-\alpha_1 \eta_{x_1})\{f\}$.
Replacing $\alpha_1$ by $q^{\alpha_1}$,  dividing both sides by $1-q$ and finally letting $q\to 1,$ we find that
$f_{x_1}=\alpha_1 f_{y_1}$, which is the same partial differential equation as that in Proposition~\ref{ppmliu}
after writing  $(x_1, y_1, \alpha_1)$ by $(x, y, \alpha)$. Thus, the case $k=1$ of Theorem~\ref{principalliuhthm} is a $q$-extension of Proposition~\ref{ppmliu}.
\end{proof}
\begin{rem}\label{remliua}
From Proposition~\ref{qliuextension}, we know that a function is a $q$-extension of $f(\alpha x+y)$ if
it can be expanded in terms of  $\Phi^{(\alpha)}_{n}(x, y|q)$.
\end{rem}

Theorem~\ref{principalliuhthm} is quite useful in $q$-series, which allows us to derive some deep $q$-formulas.
For example, with this theorem, we can prove the following curious $q$-formula:
\begin{align*}
&\frac{(q; q)_\infty}{2\pi}\int_{0}^{\pi}
\frac{h(\cos 2\theta; 1) }
{h(\cos \theta; a, b, c, u, v)} {_3\phi_2}\({{\alpha, ve^{i\theta}, ve^{-i\theta}}\atop{\alpha dv, uv}}; q, du\) d\theta\\
&=\frac{(dv; q)_\infty}
{(1-q)v(q, u/v, qv/u, \alpha dv, uv, ab, ac, bc; q)_\infty}\\
&\qquad\times{\int_u^v \frac{(qx/u, qx/v, abcx, \alpha dx; q)_\infty}{(ax, bx, cx, dx; q)_\infty}d_q x},
\end{align*}
where $\max\{|a|, |b|, |c|, |d|, |u|, |v|\}<1$, which includes the Nassrallah-Rahman integral as a special case.
This paper is organized as follows. Section~2 is devoted the proof of  Theorem~\ref{principalliuhthm}. In Section~3,
we will use Theorem~\ref{principalliuhthm} to prove the $q$-Mehler formula for the homogeneous Hahn polynomials.
Some $q$-integrals are discussed in Sections~4 and 5. A multilinear generating function for the Hahn polynomials is established in Section~6.
A proof of the Srivastava-Jain formula for the Hahn polynomials is given in Section~7.

\section{the proof of Theorem~\ref{principalliuhthm}}
To prove Theorem~\ref{principalliuhthm},  we need the following fundamental property of
several complex variables (see, for example \cite[p. 5, Proposition~ 1]{Malgrange}, \cite[p. 90]{Range}).
\begin{prop}\label{mcvarapp}
If $f(x_1, x_2, \ldots,  x_k)$ is analytic at the origin $(0, 0, \ldots,  0)\in \mathbb{C}^k$, then,
$f$ can be expanded in an absolutely convergent power series,
 \[
 f(x_1, x_2, \ldots,  x_k)=\sum_{n_1, n_2, \ldots, n_k=0}^\infty \lambda_{n_1, n_2, \ldots, n_k}
 x_1^{n_1} x_2^{n_2}\cdots x_k^{n_k}.
 \]
\end{prop}
Now we begin to prove Theorem~\ref{principalliuhthm} by using Proposition~\ref{mcvarapp}.
\begin{proof} The theorem can be proved  by induction. We first prove the theorem for the case
$k=1.$
Since $f$ is analytic at $(0, 0),$ from Proposition~\ref{mcvarapp}, we know that $f$ can be expanded in
an absolutely convergent power series in a neighborhood of $(0, 0)$. Thus there exists a sequence $\lambda_{m, n}$
independent of $x_1$ and $y_1$ such that
\begin{equation}
 f(x_1, y_1)=\sum_{m, n=0}^\infty \lambda_{m, n} x_1^my_1^n=\sum_{m=0}^\infty x_1^m \left\{\sum_{n=0}^\infty \lambda_{m, n} y_1^n\right\}.
 \label{liu:eqn1}
\end{equation}
Substituting this into the $q$-partial differential equation,
$\partial_{q, x_1}\{f(x_1, y_1)\}=\partial_{q, y_1}\{f(x_1, y_1)-\alpha f(qx_1, y_1)\},$
and using the fact,  $\partial_{q, x_1}\{x_1^m\}=(1-q^m)x_1^{m-1},$ we find that
\[
\sum_{m=1}^\infty (1-q^m)x_1^{m-1} \sum_{n=0}^\infty \lambda_{m, n} y_1^n
=\sum_{m=0}^\infty (1-\alpha q^m)x_1^m  \partial_{q, y_1} \left\{\sum_{n=0}^\infty \lambda_{m, n} y_1^n\right\}.
\]
Equating the coefficients of $x_1^{m-1}$ on both sides of the above equation, we easily deduce that
\[
\sum_{n=0}^\infty \lambda_{m, n} y_1^n=\frac{1-\alpha q^{m-1}}{1-q^m} \partial_{q, y_1}\left\{\sum_{n=0}^\infty \lambda_{m-1, n} y_1^n\right\}.
\]
Iterating the above recurrence relation $(m-1)$ times, we conclude that
\[
\sum_{n=0}^\infty \lambda_{m, n} y_1^n=\frac{(\alpha; q)_m}{(q; q)_m} \partial^n_{q, y_1}\left\{\sum_{n=0}^\infty \lambda_{0, n} y_1^n\right\}.
\]
With the help of  the identity,  ${\partial}^m_{q, y_1}\{y_1^n\}=(q; q)_n y_1^{n-m}/(q; q)_{n-m}$, we obtain
\[
\sum_{n=0}^\infty \lambda_{m, n} y_1^n=(\alpha; q)_m\sum_{n=m}^\infty \lambda_{0, n}{n\brack m}_q y_1^{n-m}.
\]
Noting that the series in (\ref{liu:eqn1}) is absolutely convergent, substituting the above equation
into (\ref{liu:eqn1})  and interchanging the order of the summation, we deduce that
\[
f(x_1, y_1)=\sum_{n=0}^\infty \lambda_{0, n} \sum_{m=0}^n {n \brack m}_q (\alpha; q)_m x_1^n y_1^{n-m}
=\sum_{n=0}^\infty \lambda_{0, n} \Phi_n^{(\alpha)}(x_1, y_1|q).
\]
Conversely, if $f(x_1, y_1)$ can be expanded in terms of $\Phi_n^{(\alpha)}(x_1, y_1|q), $
then using Proposition~\ref{liupdepp},  we find that
 $
 \partial_{q, x_1}\{f(x_1, y_1)\}=\partial_{q, y_1}\{f(x_1, y_1)-\alpha f(qx_1, y_1)\}.
 $
We conclude that Theorem~\ref{principalliuhthm} holds when $k=1$.

Now, we assume that the theorem is true for the case $k-1$ and consider the case $k$.
If we regard $f(x_1, y_1, \ldots, x_k, y_k)$ as a function of $x_1$ and $y_1, $ then $f$ is analytic at $(0, 0)$ and satisfies
$\partial_{q, x_1}\{f\}=\partial_{q, y_1}(1-\alpha_1 \eta_{x_1})\{f\}.$ Thus,  there exists a sequence
$\{c_{n_1}(x_2, y_2, \ldots, x_k, y_k)\}$ independent of $x_1$ and $y_1$ such that
\begin{equation}
f(x_1, y_1, \ldots, x_k, y_k)=\sum_{n_1=0}^\infty c_{n_1}(x_2, y_2, \ldots, x_k, y_k)\Phi^{(\alpha_1)}_{n_1}(x_1, y_1|q).
\label{liu:eqn2}
\end{equation}
Setting $x_1=0$ in the above equation and using $=\Phi^{(\alpha_1)}_{n_1}(0, y_1|q)=y_1^{n_1},$ we obtain
\[
f(0, y_1, x_2, y_2,  \ldots, x_k, y_k)=\sum_{n_1=0}^\infty c_{n_1}(x_2, y_2, \ldots, x_k, y_k)y_1^{n_1}.
\]
Using the Maclaurin expansion theorem, we immediately deduce that
\[
c_{n_1}(x_2, y_2, \ldots, x_k, y_k)=\frac{\partial^{n_1} f(0, y_1, x_2, y_2,  \ldots, x_k, y_k)}{n_1! \partial {y_1}^{n_1}}\Big|_{y_1=0}
\]
Since $f(x_1, y_1, \ldots, x_k, y_k)$ is analytic near $(x_1, y_1, \ldots, x_k, y_k)=(0, \ldots, 0)\in \mathbb{C}^{2k},$  from
the above equation, we know that $c_{n_1}(x_2, y_2, \ldots, x_k, y_k)$ is analytic near $(x_2, y_2, \ldots, x_k, y_k)=(0, \ldots, 0)\in \mathbb{C}^{2k-2}.$   Substituting (\ref{liu:eqn2}) into the $q$-partial equations in Theorem~\ref{principalliuhthm}.
we find, for $j=2, \ldots k$, that
\begin{align*}
&\sum_{n_1=0}^\infty \partial_{q, x_j}\{c_{n_1}(x_2, y_2, \ldots, x_k, y_k)\} \Phi^{(\alpha_1)}_{n_1}(x_1, y_1|q)\\
&=\sum_{n_1=0}^\infty \partial_{q, y_j}(1-\alpha_j \eta_{x_j})\{c_{n_1}(x_2, y_2, \ldots, x_k, y_k)\} \Phi^{(\alpha_1)}_{n_1}(x_1, y_1|q).
\end{align*}
By equating the coefficients of $\Phi^{(\alpha_1)}_{n_1}(x_1, y_1|q)$ in the above equation, we find that
for $j=2, \ldots,  k,$
\[
\partial_{q, x_j}\{c_{n_1}(x_2, y_2, \ldots, x_k, y_k)\}
=\partial_{q, y_j}(1-\alpha_j \eta_{x_j})\{c_{n_1}(x_2, y_2, \ldots, x_k, y_k)\}.
\]
Thus by the inductive hypothesis, there exists a sequence $\lambda_{n_1, n_2, \ldots, n_k}$ independent of
$x_2, y_2, \ldots, x_k, y_k$ (of course independent of $x_1$ and $y_1$) such that
\[
c_{n_1}(x_2, y_2, \ldots, x_k, y_k)=\sum_{n_2, \ldots, n_k=0}^\infty \alpha_{n_1, n_2, \ldots, n_k}
\Phi^{(\alpha_2)}_{n_2}(x_2, y_2|q)\ldots \Phi^{(\alpha_k)}_{n_k}(x_k, y_k|q).
\]
Substituting this equation into (\ref{liu:eqn2}), we find that $f$ can be expanded
in terms of $\Phi^{(\alpha_1)}_{n_1}(x_1, y_1|q)\cdots \Phi^{(\alpha_k)}_{n_k}(x_k, y_k|q).$

Conversely, if $f$ can be expanded in terms of $\Phi^{(\alpha_1)}_{n_1}(x_1, y_1|q)\cdots \Phi^{(\alpha_k)}_{n_k}(x_k, y_k|q),$
then using Proposition~\ref{liupdepp},  we find that $\partial_{q, x_j}\{f\}=\partial_{q, y_j}(1-\alpha_j \eta_{x_j}\{f\}$ for
$j=1, 2, \ldots, k.$ This completes the proof of Theorem~~\ref{principalliuhthm}.
\end{proof}
To determine if a given function is an analytic functions in several complex variables,
we often use the following theorem (see, for example, \cite[p. 28]{Taylor}).
\begin{thm}\label{hartogthm} {\rm (Hartogs' theorem).}
If a complex valued function $f(z_1, z_2, \ldots, z_n)$ is holomorphic (analytic) in each variable separately in a domain $U\in\mathbb{C}^n,$
then,  it is holomorphic (analytic) in $U.$
\end{thm}
\section{Generating functions for the homogeneous Hahn polynomials }
In this section we use Theorem~~\ref{principalliuhthm} to prove the
following two known results to explain our method.
\begin{thm} \label{genthma}If $\max\{|xt|, |yt|\}<1,$ then,  we have
\begin{equation}
\sum_{n=0}^\infty \Phi_n^{(\alpha)} (x, y|q) \frac{t^n}{(q; q)_n}
=\frac{(\alpha xt; q)_\infty}{(xt, yt; q)_\infty}.
\label{genfun:eqn1}
\end{equation}
\end{thm}
\begin{proof} Denote the right-hand side of (\ref{genfun:eqn1}) by $f(x, y).$ It is easily seen that $f(x, y)$ is
an analytic function of $x$ and $y$, for $\max\{|xt|, |yt|\}<1.$  Thus, $f(x, y)$ is analytic at $(0, 0)\in \mathbb{C}^2.$
A direct computation shows that
\[
\partial_{q, x}\{f(x, y)\}=\partial_{q, y}\{f(x, y)-\alpha f(qx, y)\}.
\]
 Thus, by Theorem~~\ref{principalliuhthm},  there exists a sequence $\{\lambda_n\}$ independent of $x$ and $y$ such that
 \[
\frac{(\alpha xt; q)_\infty}{(xt, yt; q)_\infty}=\sum_{n=0}^\infty \lambda_n \Phi_n^{(\alpha)}(x, y|q).
 \]
 Putting $x=0$ in the above equation, using $\Phi_n^{(\alpha)}(0, y|q)=y^n$ and $q$-binomial theorem, we find that
 \[
 \sum_{n=0}^\infty \lambda_n y^n=\frac{1}{(yt; q)_\infty}=\sum_{n=0}^\infty \frac{y^nt^n}{(q; q)_n},
 \]
 which implies $\lambda_n=t^n/(q; q)_n.$ We complete the proof of Theorem~\ref{genthma}.
\end{proof}
Al-Salam and Carlitz \cite{SalamCarlitz} (see also \cite {Cao2010, Liu1997}) found the following two bilinear generating functions
by the transformation theory of $q$-series, which is also called the $q$-Mehler formula for  the homogeneous Hahn polynomials.
\begin{thm}\label{genthmb} For $\max\{|xtu|, |xtv|, |ytu|, |ytv| \}<1,$ we have
\begin{align}
&\sum_{n=0}^\infty \Phi^{(\alpha)}_n (x, y|q)\Phi^{(\beta)}_n (u, v|q)\frac{t^n}{(q; q)_n}\label{genfun:eqn2}\\
&=\frac{(\alpha xtv, \beta ytu; q)_\infty}{(xtv, ytv, ytu; q)_\infty}
{_3\phi_2}\({{\alpha, \beta, ytv}\atop{\alpha xtv, \beta ytu}}; q, xtu\).\nonumber
\end{align}
\end{thm}
Now we will use Theorem~~\ref{principalliuhthm} to give a new  proof of this formula.
\begin{proof}  If we use $f(x, y)$ to denote the right-hand side of (\ref{genfun:eqn2}), then it is easy to show that
$f(x, y)$ is a two-variable analytic function of $x$ and $y$ for  $\max\{|xtu|, |xtv|, |ytu|, |ytv| \}<1.$
By a direct computation,  we find that
\begin{align*}
&\partial_{q, x}\{f(x, y)\}=\partial_{q, y}\{f(x, y)-\alpha f(qx, y)\}\\
&=\frac{t(\alpha xtv, \beta ytu; q)_\infty}{(xtv, ytv, ytu; q)_\infty}
\sum_{n=0}^\infty \frac{(\alpha; q)_{n+1} (\beta, ytv; q)_n (xtu)^n}{(q; q)_n (\alpha xtv, \beta ytu; q)_{n+1}}
\(u+(v-\beta u-ytuv) q^n\).
\end{align*}
Thus, by Theorem~\ref{principalliuhthm},  there exists a sequence $\{\lambda_n\}$ independent of $x$ and $y$ such
that
\begin{equation}
f(x, y)=\sum_{n=0}^\infty \lambda_n \Phi_n^{(\alpha)} (x, y|q).
\label{genfun:eqn3}
\end{equation}
Setting $x=0$ in the above equation and using $\Phi_n^{(\alpha)}(0, y|q)=y^n$, we immediately find that
\[
\sum_{n=0}^\infty \lambda_n y^n=\frac{(\beta ytu; q)_\infty}{(ytu, ytv; q)_\infty}.
\]
Using Theorem~\ref{genthma}, we find that the right-hand side of the above equation equals
\[
\sum_{n=0}^\infty \Phi_n^{(\beta)}(u, v|q)\frac{(ty)^n}{(q; q)_n}.
\]
It follows that $\lambda_n=\Phi^{(\beta)}(u, v|q)t^n/(q; q)_n.$
Substituting this into (\ref{genfun:eqn3}), we complete the proof of Theorem~\ref{genthmb}.
\end{proof}
\section{some $q$-integral formulas}
 The Jackson $q$-integral of the function $f(x)$ from $a$ to $b$ is defined as
\[
\int_{a}^b f(x) d_q x=(1-q)\sum_{n=0}^\infty [b f(bq^n)-af(aq^n)]q^n.
\]
The Andrews-Askey integral formula \cite[Theorem~1]{Andrews+Askey} can be stated in the following proposition.
\begin{prop}\label{andaskpp} If there are no zero factors in the denominator of the integral, then,
we have
\begin{equation}
\int_{u}^v \frac{(qx/u, qx/v; q)_\infty}{(bx, cx; q)_\infty}d_q x
=\frac{(1-q)v(q, u/v, qv/u, bcuv; q)_\infty}{(bu, bv, cu, cv; q)_\infty}.
\label{raaeqn1}
\end{equation}
\end{prop}
\begin{defn}\label{carlitzdefn} For any nonnegative integer $n,$ we define $W_n(a, b, u, v|q)$ as
\[
W_n(a, b, u, v|q)=\sum_{j=0}^n {n\brack j}_q \frac{(av, bv; q)_j}{(abuv; q)_j}u^jv^{n-j}.
\]
\end{defn}
Applying $\mathcal{\partial}_{q, b}^n$ to act both sides of (\ref{raaeqn1}) and then using the $q$-Leibniz rule, one
can easily find the following proposition \cite{Wang2009}.
\begin{prop}\label{wangpp}
 If there are no zero factors in the denominator of the integral, then,
we have
\begin{equation*}
\int_{u}^v \frac{x^n(qx/u, qx/v; q)_\infty}{(bx, cx; q)_\infty}d_q x
=\frac{(1-q)v(q, u/v, qv/u, bcuv; q)_\infty}{(bu, bv, cu, cv; q)_\infty}
W_n(b, c, u, v|q).
\end{equation*}
\end{prop}
The main result of this section is the following $q$-integral formula.
\begin{thm}\label{liuqint}If there are no zero factors in the denominator of the integral, then,
we have
\end{thm}
\begin{align}
&\int_{u}^v \frac{(qx/u, qx/v, \alpha ax; q)_\infty}{(ax, bx, cx, dx; q)_\infty} d_q x
\label{raaeqn2}\\
&=\frac{(1-q)v(q, u/v, qv/u, cduv; q)_\infty }{(cu, cv, du, dv; q)_\infty }\sum_{n=0}^\infty \frac{W_n(c, d, u, v|q) \Phi_n^{(\alpha)}(a, b|q)}{(q; q)_n}.\nonumber
\end{align}
\begin{proof}
If we denote the $q$-integral in the above theorem by $I(a, b)$, then it is
easy to show that $I(a, b)$ is analytic near $(0, 0)\in \mathbb {C}^2.$ A straightforward
evaluation shows that
\[
\partial_{q, a} \{I(a, b)\}=\partial_{q, b}\{I(a, b)-\alpha I(aq, b)\}
=(1-\alpha)\int_{u}^v \frac{x(qx/u, qx/v, \alpha qax; q)_\infty}{(ax, bx, cx, dx; q)_\infty} d_q x.
\]
Thus, by Theorem~\ref{principalliuhthm}, there exists a sequence $\{\lambda_n\}$ independent of $a$ and $b$ such that
\begin{equation}
I(a, b)=\int_{u}^v \frac{(qx/u, qx/v, \alpha ax; q)_\infty}{(ax, bx, cx, dx; q)_\infty} d_q x=
\sum_{n=0}^\infty \lambda_n \Phi_n^{(\alpha)} (a, b|q).
\label{raaeqn3}
\end{equation}
Setting $a=0$ in the above equation and using $\Phi_n^{(\alpha)} (0, b|q)=b^n$, we immediately deduce that
\[
\int_{u}^v \frac{(qx/u, qx/v; q)_\infty}{( bx, cx, dx; q)_\infty} d_q x=
\sum_{n=0}^\infty \lambda_n b^n.
\]
Applying $\mathcal{\partial}_{q, b}^n$ to act both sides of the above equation, setting $b=0$,
and using Proposition~\ref{wangpp}, we obtain
\[
\lambda_n=\frac{(1-q)v(q, u/v, qv/u, cduv; q)_\infty W_n(c, d, u, v|q)}{(cu, cv, du, dv; q)_\infty (q; q)_n}.
\]
Substituting the above equation into (\ref{raaeqn3}), we complete the proof Theorem~\ref{liuqint}.
\end{proof}
Theorem~\ref{liuqint} implies some interesting special cases. For example, using this theorem, we can recover
the following $q$-integral formula \cite[Theorem~9]{Liu2010}.

\begin{thm}\label{liuint}If there are no zero factors in the denominator of the integral, then,
we have
\end{thm}
\begin{align*}
\int_{u}^v \frac{(qx/u, qx/v, \alpha ax; q)_\infty}{(ax, bx, cx; q)_\infty} d_q x&=
\frac{(1-q)v (q, u/v, qv/u, \alpha a v, bcuv; q)_\infty}{(av, bu, bv, cu, cv; q)_\infty}\\
&\qquad\times {_3\phi_2}\({{\alpha, bv, cv}\atop{\alpha av, bcuv}}; q, au\).
\end{align*}
\begin{proof}
Setting $d=0$ in (\ref{raaeqn2}) and noticing that $W_n(c, 0, u, v|q)=\Phi^{(cv)}(u, v|q),$ we
find that the $q$-integral in Theorem~\ref{liuint} equals
\[
\frac{(1-q)v(q, u/v, qv/u; q)_\infty}{( cu, cv; q)_\infty}
\sum_{n=0}^\infty \frac{\Phi_n^{(\alpha)} (a, b|q)\Phi_n^{(cv)} (u, v|q)}{(q; q)_n}.
\]
Using Theorem~\ref{genthmb}, we find that the series on the right-hand side of the above equation equals
\[
\frac{(\alpha av, bcuv; q)_\infty}{(av, bu, bv; q)_\infty}
{_3\phi_2}\({{\alpha, bv, cv}\atop{\alpha av, bcuv}}; q, au\).
\]
Combining the above two equations, we complete the proof of the theorem.
\end{proof}
Theorem~\ref{liuint} includes the following $q$-integral formula due to Al-Salam and Verma \cite[Eq. (1.3)]{SalamVerma} as a special case.
\begin{cor}\label{liuinta}  If there are no zero factors in the denominator of the integral, then,
we have
\[
\int_{u}^v \frac{(qx/u, qx/v, abcuvx; q)_\infty}{(ax, bx, cx; q)_\infty} d_q x=
\frac{(1-q)v (q, u/v, qv/u, abuv, acuv,  bcuv; q)_\infty}{(au, av, bu, bv, cu, cv; q)_\infty}.
\]
\end{cor}
\begin{proof} Taking $\alpha=bcuv$ in Theorem~\ref{liuint}, we immediately find that
\begin{align*}
\int_{u}^v \frac{(qx/u, qx/v, abcuvx; q)_\infty}{(ax, bx, cx; q)_\infty} d_q x&=
\frac{(1-q)v (q, u/v, qv/u, abcuv^2, bcuv; q)}{(av, bu, bv, cu, cv; q)_\infty}\\
&\qquad\times {_2\phi_1}\({{ bv, cv}\atop{abcuv^2}}; q, au\).
\end{align*}
Using the $q$-Gauss summation, we find that
\[
{_2\phi_1}\({{ bv, cv}\atop{abcuv^2}}; q, au\)=\frac{(abuv, acuv; q)_\infty}{(abcuv^2, au; q)_\infty}.
\]
Combining the above two equations, we complete the proof of Corollary~\ref{liuinta}.
\end{proof}
Using the same argument that we used to prove Theorem~\ref{liuqint}, we can derive the
following $q$-integral formula.
\begin{thm}\label{liuqqthm}
 If there are no zero factors in the denominator of the integral, then,
we have
\begin{align*}
&\int_{u}^v \frac{(qx/u, qx/v, \alpha ax, \beta cx; q)_\infty}{(ax, bx, cx, dx; q)_\infty} d_q x\\
&=(1-q)v (q, u/v, qv/u; q)_\infty \sum_{m, n=0}^\infty \Phi^{(\alpha)}_m(a, b|q)\Phi^{(\beta)}_n(c, d|q)
\frac{h_{m+n}(u, v|q)}{(q; q)_m (q; q)_n}.
\end{align*}
\end{thm}

\section{Some applications of Theorem~\ref{liuint}}
In this section we will discuss some applications of Theorem~\ref{liuint} to $q$-beta integrals. We begin with the
following proposition.
\begin{prop}\label{qintpp} If $f(x)$ is a function of $x$ such that
$
\lim_{n\to \infty}|{f(xq^n)}/{f(xq^{n-1})}|\le 1,
$
then, for  $\max\{|u|, |v|\}<1,$  the $q$-integral
\begin{equation}
\int_{u}^v \frac{(qx/u, qx/v; q)_\infty f(x)}{(x e^{i\theta}, xe^{-i\theta}; q)_\infty}d_q x,
\label{qinteqn1}
\end{equation}
converges absolutely and uniformly on $\theta \in [0, \pi]$, provided that
there are no zero factors in the denominator of the $q$-integral.
\end{prop}
\begin{proof} From the definition of the $q$-integral, we know that the $q$-integral in
(\ref{qinteqn1}) can be written in the form
\begin{align*}
&\frac{(1-q)v (qv/u, q; q)_\infty}{(v e^{i\theta}, v e^{-i\theta}; q)_\infty}
\sum_{n=0}^\infty \frac{(v e^{i\theta}, v e^{-i\theta}; q)_n f(q^n v)q^n}{(q, qv/u; q)_n}\\
&-\frac{(1-q)u (qu/v, q; q)_\infty}{(u e^{i\theta}, u e^{-i\theta}; q)_\infty}
\sum_{n=0}^\infty \frac{(u e^{i\theta}, u e^{-i\theta}; q)_n f(q^n u)q^n}{(q, qu/v; q)_n}.
\end{align*}
It is easily seen that $|(z e^{i\theta}, z e^{-i\theta}; q)_n|\le (-|z|; q)_n^2$, and
using the $q$-binomial theorem, we know that  for $ |z|<1,$
\[
\Big| \frac{1}{(z e^{i\theta}, z e^{-i\theta}; q)_\infty}\Big|\le \frac{1}{(|z|; q)^2_\infty}.
\]
Thus, we have
\begin{align*}
\Big|\int_{u}^v \frac{(qx/u, qx/v; q)_\infty f(x)}{(x e^{i\theta}, xe^{-i\theta}; q)_\infty}d_q x\Big|
&\le \frac{(1-q)|v||(qv/u; q)_\infty|}{(|v|; q)_\infty^2}
\sum_{n=0}^\infty \frac{(-|v|; q)_n^2 |f(q^nv)|q^n}{|(q, qv/u; q)_n|}\\
&+\frac{(1-q)|u||(qu/v; q)_\infty|}{(|u|; q)_\infty^2}
\sum_{n=0}^\infty \frac{(-|u|; q)_n^2 |f(q^nu)|q^n}{|(q, qu/v; q)_n|}.
\end{align*}
Using the ratio test, we find that the series on the right-hand side of the above equation is convergent
for $\max\{|u|, |v|\}<1.$ Thus, the series in (\ref{qinteqn1}) converges absolutely and uniformly on $\theta \in [0, \pi]$.
\end{proof}
\begin{defn}\label{defn1}
For$\ x=\cos \T$, we define $h(x; a)$ and
 $ h(x; a_1, a_2, \ldots, a_m)$  as follows
 \begin{align*}
&h(x; a)=(ae^{i\T}, ae^{-i\T}; q)_\infty=\prod_{k=0}^\infty (1-2q^k ax+q^{2k}a^2) \\
&h(x; a_1, a_2, \ldots, a_m)=h(x; a_1)h(x; a_2)\cdots h(x; a_m).
 \end{align*}
\end{defn}
The following important integral evaluation is due to Askey and Wilson (see, for example \cite[p. 154]{Gas+Rah}).
\begin{prop}\label{ppasw}  For $\max\{|a|, |b|, |c|, |d|\}<1$, we have the integral formula
\begin{equation*}
I(a, b, c, d):= \int_{0}^{\pi} \frac{h(\cos 2\T; 1)d\T}{h(\cos \T; a, b, c,  d)}
=\frac{2\pi (abcd; q)_\infty}{(q, ab, ac, ad, bc, bd, cd;q)_\infty}.
\end{equation*}
\end{prop}

The main result of this section is the following general $q$-integral identity.
\begin{thm}\label{aaliuthm} If $\max\{|a|, |b|, |c|,  |u|, |v|\}<1,$ and $f(x)$ is a
function of $x$ such that $\lim_{n\to \infty}|{f(xq^n)}/{f(xq^{n-1})}|\le 1$,   then, we have
\begin{align*}
&\frac{(q; q)_\infty}{2\pi}\int_{0}^{\pi} \frac{h(\cos 2\theta; 1)}
{h(\cos \theta; a, b,c)} \left\{\int_{u}^v \frac{(qx/u, qx/v; q)_\infty f(x)}{(x e^{i\theta}, xe^{-i\theta}; q)_\infty}d_q x\right\} d\theta\\
&=\frac{1}{(ab, ac, bc; q)_\infty}\int_{u}^v \frac{(qx/u, qx/v, abcx; q)_\infty f(x)}{(ax, bx, cx; q)_\infty}d_q x.
\end{align*}
\end{thm}
\begin{proof}
Letting $d=x$ in the Askey--Wilson integral and  then multiplying both sides by $f(x)(qx/u, qx/v; q)_\infty,$ we obtain
\begin{align}
&\frac{(q; q)_\infty}{2\pi}\int_{0}^{\pi} \frac{h(\cos 2\theta; 1)(qx/u, qx/v; q)_\infty f(x)  d\theta}
{h(\cos \theta; a, b,c)(x e^{i\theta}, xe^{-i\theta}; q)_\infty} \label{qinteqn2}\\
&=\frac{(qx/u, qx/v, abcx; q)_\infty f(x)}{(ab, ac, bc, ax, bx, cx; q)_\infty}.\nonumber
\end{align}
For the sake of brevity, we temporarily use the compact symbol $A_k(u, v, \theta)$ to denote
\[
\frac{(q^{k+1}u/v, q^{k+1}; q)_\infty f(uq^k)q^k}{( uq^k e^{i\theta}, uq^k e^{-i\theta}; q)_\infty}.
\]
Putting $x=uq^k$ in (\ref{qinteqn2}), multiplying both sides by $q^k$,   and summing the resulting equation
over $0\le k <\infty,$ we obtain
\begin{align}
&\frac{(q; q)_\infty}{2\pi}\sum_{k=0}^\infty \int_{0}^{\pi} \frac{h(\cos 2\theta; 1)A_k(u, v, \theta) d\theta}
{h(\cos \theta; a, b,c)} \label{qinteqn3} \\
&=\frac{1}{(ab, ac, bc; q)_\infty}\sum_{k=0}^\infty \frac{(q^{k+1}u/v, q^{k+1}, abcuq^k; q)_\infty f(uq^k)q^k}{( auq^k, buq^k, cuq^k; q)_\infty}.
\nonumber
\end{align}
Using the triangle inequality for complex numbers,  we easily find that $|h(\cos \theta; 1)|\le(-1; q)_\infty,$ and by the $q$-binomial
theorem, we can find that for $\max\{|a|, |b|, |c|\}<1,$
\[
\Big|\frac{1}{h(\cos \theta; a, b, c)}\Big|\le \frac{1}{(|a|, |b|, |c|; q)_\infty^2}.
\]
By Proposition~\ref{qintpp}, we know that $\sum_{k \ge 0} A_k(u, v, \theta)$ is absolutely and uniformly convergent on
$\theta \in [0, \pi]$. Thus, we can interchange the order of summation and the integral in (\ref{qinteqn3}) to find that
\begin{align}
&\frac{(q; q)_\infty}{2\pi} \int_{0}^{\pi} \frac{h(\cos 2\theta; 1) }{h(\cos \theta; a, b,c)}
\(\sum_{k=0}^\infty A_k(u, v, \theta)\)d\theta
 \label{qinteqn4}\\
&=\frac{1}{(ab, ac, bc; q)_\infty}\sum_{k=0}^\infty \frac{(q^{k+1}u/v, q^{k+1}, abcuq^k; q)_\infty f(uq^k)q^k}{( auq^k, buq^k, cuq^k; q)_\infty}.
\nonumber
\end{align}
Interchanging $u$ and $v$ in the above equation, we conclude that
\begin{align}
&\frac{(q; q)_\infty}{2\pi} \int_{0}^{\pi} \frac{h(\cos 2\theta; 1) }{h(\cos \theta; a, b,c)}
\(\sum_{k=0}^\infty A_k(v, u, \theta)\)d\theta
 \label{qinteqn5}\\
&=\frac{1}{(ab, ac, bc; q)_\infty}\sum_{k=0}^\infty \frac{(q^{k+1}v/u, q^{k+1}, abcvq^k; q)_\infty f(vq^k)q^k}{( avq^k, bvq^k, cvq^k; q)_\infty}.
\nonumber
\end{align}
Multiplying (\ref{qinteqn4}) by $-u(1-q)$ and multiplying (\ref{qinteqn5}) by $v(1-q)$, then adding the two
resulting equations together and using the definition of $q$-integral, we complete the proof of the theorem.
\end{proof}
Using Theorem~\ref{aaliuthm}, we can derive the following interesting integral formula.
\begin{thm}\label{liuqintthm} For $\max\{|a|, |b|, |c|, |d|,  |u|, |v|\}<1,$
we have
\begin{align*}
&\frac{(q; q)_\infty}{2\pi}\int_{0}^{\pi}
\frac{h(\cos 2\theta; 1) }
{h(\cos \theta; a, b, c, u, v)} {_3\phi_2}\({{\alpha, ve^{i\theta}, ve^{-i\theta}}\atop{\alpha dv, uv}}; q, du\) d\theta\\
&=\frac{(dv; q)_\infty}
{(1-q)v(q, u/v, qv/u, \alpha dv, uv, ab, ac, bc; q)_\infty}\\
&\qquad\times{\int_u^v \frac{(qx/u, qx/v, abcx, \alpha dx; q)_\infty}{(ax, bx, cx, dx; q)_\infty}d_q x}.
\end{align*}
\end{thm}
\begin{proof} Choosing  $f(x)=(\alpha dx; q)_\infty/(dx; q)_\infty$ in Theorem~\ref{aaliuthm}, we find that
\begin{align*}
&\frac{(q; q)_\infty}{2\pi}\int_{0}^{\pi} \frac{h(\cos 2\theta; 1)}
{h(\cos \theta; a, b,c)} \left\{\int_{u}^v \frac{(qx/u, qx/v, \alpha dx; q)_\infty }{(dx, x e^{i\theta}, xe^{-i\theta}; q)_\infty}d_q x\right\} d\theta\\
&=\frac{1}{(ab, ac, bc; q)_\infty}\int_{u}^v \frac{(qx/u, qx/v, abcx, \alpha dx; q)_\infty }{(ax, bx, cx, dx; q)_\infty}d_q x.
\end{align*}
If $(a, b, c)$ is replaced by $(d, e^{i\theta}, e^{-i\theta})$ in Theorem~\ref{liuint}, we conclude that
\begin{align*}
\int_{u}^v \frac{(qx/u, qx/v, \alpha dx; q)_\infty}{(dx, xe^{i\theta}, xe^{-i\theta}; q)_\infty} d_q x&=
\frac{(1-q)v (q, u/v, qv/u, \alpha d v, uv; q)}{(dv; q)_\infty h(\cos \theta; u, v)}\\
&\qquad\times {_3\phi_2}\({{\alpha, ve^{i\theta}, ve^{-i\theta}}\atop{\alpha dv, uv}}; q, du\).
\end{align*}
Combining the above two equations, we complete the proof of Theorem~\ref{liuqintthm}.
\end{proof}
Theorem~\ref{liuqintthm}  implies the Ismail-Stanton-Viennot integral (see, for example \cite{Liu1997}) as a special case.
\begin{prop}\label{isvintpp} For $\max\{|a|, |b|, |c|, |u|, |v|\}<1,$
we have
\begin{align*}
&\frac{(q; q)_\infty}{2\pi}\int_{0}^{\pi}
\frac{h(\cos 2\theta; 1) d\theta}
{h(\cos \theta; a, b, c, u, v)}\\
&=\frac{( abcv,  bcuv; q)_\infty}
{(ab, ac, bc, av, bu, bv, cu, cv, uv; q)_\infty}
{_3\phi_2}\({{bc, bv, cv}\atop{abcv, bcuv}}; q, au\).
\end{align*}
\end{prop}
\begin{proof} Taking $\alpha=1$ in Theorem~\ref{liuqintthm}, we immediately find that
\begin{align*}
&\frac{(q; q)_\infty}{2\pi  }
\int_{0}^{\pi}
\frac{h(\cos 2\theta; 1) d\theta}
{h(\cos \theta; a, b, c, u, v)}\\
&=\frac{1}{(1-q)v(q, u/v, qv/u, ab, ac, bc, uv; q)_\infty}\int_{u}^v \frac{(qx/u, qx/v, abcx; q)_\infty }{(ax, bx, cx; q)_\infty}d_q x.
\end{align*}
Taking $\alpha=bc$ in Theorem~\ref{liuint}, we are led to the $q$-integral formula.
\begin{align*}
\int_{u}^v \frac{(qx/u, qx/v, abcx; q)_\infty}{(ax, bx, cx; q)_\infty} d_q x&=
\frac{(1-q)v (q, u/v, qv/u, abc v, bcuv; q)}{(av, bu, bv, cu, cv; q)_\infty}\\
&\qquad\times {_3\phi_2}\({{bc, bv, cv}\atop{abcv, bcuv}}; q, au\).
\end{align*}
Combining the above two equations, we finish the proof of Proposition~\ref{isvintpp}.
\end{proof}
Theorem~\ref{liuqintthm} also includes  the following $q$-integral as a special case.
\begin{prop}\label{liubetapp}  For $\max\{|a|, |b|, |c|, |d|,  |u|, |v|\}<1,$ we have the identity
\begin{align*}
&\frac{(1-q)v(q, q,  u/v, qv/u, uv; q)_\infty}{2\pi  (du, dv; q)_\infty}
\int_{0}^{\pi}
\frac{h(\cos 2\theta; 1)h(\cos \theta; duv) d\theta}
{h(\cos \theta; a, b, c, u, v)}\\
&=\frac{1}{(ab, ac, bc; q)_\infty}\int_{u}^v \frac{(qx/u, qx/v, abcx, duvx; q)_\infty }{(ax, bx, cx, dx; q)_\infty}d_q x.
\end{align*}
\end{prop}
\begin{proof}
Taking $\alpha=uv$ in Theorem~\ref{liuqintthm}, and then using the $q$-Gauss summation in the resulting equation,
we complete the proof of Theorem~\ref{liubetapp}.
\end{proof}
Taking $d=abc$ in Proposition~\ref{liubetapp} and then using Corollary~\ref{liuinta}, we are led to the
Nassrallah-Rahman integral (see, for example \cite[p. 159]{Gas+Rah}).
\begin{prop}\label{nrintpp} For $\max\{|a|, |b|, |c|, |u|, |v|\}<1,$
we have
\begin{align*}
&\frac{(q; q)_\infty}{2\pi}\int_{0}^{\pi}
\frac{h(\cos 2\theta; 1)h(\cos \theta; abcuv) d\theta}
{h(\cos \theta; a, b, c, u, v)}\\
&=\frac{(abcu, abcv, abuv, acuv, bcuv; q)_\infty}
{(ab, ac, au, av, bc, bu, bv, cu, cv, uv; q)_\infty}.
\end{align*}
\end{prop}

\section{A multilinear generating function for the Hahn polynomials }
\begin{thm} \label{mghahn} If $\max\{ |a|, |c|, |x_1|, |y_1|, \ldots, |x_k|, |y_k|\}<1,$ then, we have the
following multilinear generating function for the Hahn polynomials:
\begin{align*}
\sum_{n_1, n_2, \ldots, n_k=0}^\infty \frac{(a; q)_{n_1+n_2+\cdots+n_k}\Phi^{(\alpha_1)}_{n_1}(x_1, y_1|q) \Phi^{(\alpha_2)}_{n_2}(x_2, y_2|q)\cdots \Phi^{(\alpha_k)}_{n_k}(x_k, y_k|q)}
{(c; q)_{n_1+n_2+\cdots+n_k}(q; q)_{n_1}(q; q)_{n_2}\cdots (q; q)_{n_k}}\\
=\frac{(a, \alpha_1 x_1, \ldots \alpha_k x_k; q)_\infty}{(c, x_1, y_1, x_2, y_2, \ldots, x_k, y_k)_\infty}
{_{2k+1}\phi_{2k}}\left({{c/a, x_1, y_1, x_2, y_2, \ldots, x_k, y_k}
\atop{\alpha_1 x_1, \alpha_2x_2, \ldots, \alpha_k x_k, 0, 0, \ldots, 0}}; q, a\right).
\end{align*}
\end{thm}
\begin{proof}
If we use  $f(x_1, y_1, \ldots, x_k, y_k)$ to denote the right-hand side of the above equation, then,
using the ratio test, we find that $f$ is an analytic function of
$u_1, v_1, \ldots, u_s, v_s$ for $\max\{|x_1|, |y_1|, \ldots, |x_k|, |y_k|\}<1.$
By a direct computation, we deduce that for $j=1, 2\ldots, k$,
\begin{align*}
\partial_{q, x_j}\{f\}&=\partial_{q, y_j}(1-\alpha_j \eta_{x_j})\{f\}\\
&=\frac{(1-\alpha_j)(a, \alpha_1 x_1, \ldots, q\alpha_j x_j, \ldots,  \alpha_k x_k; q)_\infty}{(c, x_1, y_1, x_2, y_2, \ldots, x_k, y_k)_\infty}\\
&\quad\times{_{2k+1}\phi_{2k}}\left({{c/a, x_1, y_1, x_2, y_2, \ldots, x_k, y_k}
\atop{\alpha_1 x_1, \ldots, q\alpha_jx_j, \ldots,  \alpha_k x_k, 0, 0, \ldots, 0}}; q, qa\right).
\end{align*}
Thus, by  Theorem~\ref{principalliuhthm},  there exists a sequence $\{\lambda_{n_1, n_2, \ldots, n_k}\}$ independent
of $x_1, y_1, \ldots, x_k, y_k$ such that $f(x_1, y_1, x_2, y_2, \ldots, x_k, y_k)$ equals
\begin{equation}
\sum_{n_1, n_2, \ldots, n_k=0}^\infty \lambda_{n_1, n_2, \ldots, n_k}
\Phi^{(\alpha_1)}_{n_1}(x_1, y_1|q) \Phi^{(\alpha_2)}_{n_2}(x_2, y_2|q)\cdots \Phi^{(\alpha_k)}_{n_k}(x_k, y_k|q).
\label{mgeneqn}
\end{equation}
Setting $x_1=x_2\cdots=x_k=0$ in the above equation, we immediately obtain
\begin{align}
&\sum_{n_1, n_2, \ldots, n_k=0}^\infty \lambda_{n_1, n_2, \ldots, n_k}
y_{1}^{n_1}y_2^{n_2}\cdots y_{k}^{n_k}\label{mgeneqn1}\\
&=\frac{(a; q)_\infty}{(c, y_1, y_2, \ldots, y_k)_\infty}
{_{k+1}\phi_{k}}\left({{c/a, y_1, y_2, \ldots, y_k}
\atop{0, 0, \ldots, 0}}; q, a\right).\nonumber
\end{align}
Andrews \cite {Andrews1972} has proved the following formula for the $q$-Lauricella function:
\begin{align*}
\sum_{n_1, n_2, \ldots, n_k=0}^\infty \frac{(a; q)_{n_1+n_2+\cdots+n_k}(b_1; q)_{n_1}(b_2; q)_{n_2}\cdots (b_k; q)_{n_k}
y_1^{n_1} y_2^{n_2}\cdots y_k^{n_k}}
{(c; q)_{n_1+n_2+\cdots+n_k}(q; q)_{n_1}(q; q)_{n_2}\cdots (q; q)_{n_k}}\\
=\frac{(a, b_1y_1, b_2y_2, \ldots, b_ky_k; q)_\infty}{(c, y_1, y_2, \ldots, y_k)_\infty}
{_{k+1}\phi_{k}}\left({{c/a, y_1, y_2, \ldots, y_k}
\atop{by_1, by_2, \ldots, by_k}}; q, a\right).
\end{align*}
Taking $b_1=b_2=\cdots=b_k=0$ in the above equation, we conclude that
\begin{align*}
\sum_{n_1, n_2, \ldots, n_k=0}^\infty \frac{(a; q)_{n_1+n_2+\cdots+n_k}
y_1^{n_1} y_2^{n_2}\cdots y_k^{n_k}}
{(c; q)_{n_1+n_2+\cdots+n_k}(q; q)_{n_1}(q; q)_{n_2}\cdots (q; q)_{n_k}}\\
=\frac{(a; q)_\infty}{(c, y_1, y_2, \ldots, y_k)_\infty}
{_{k+1}\phi_{k}}\left({{c/a, y_1, y_2, \ldots, y_k}
\atop{0, 0, \ldots, 0}}; q, a\right).
\end{align*}
Equating the above equation and (\ref{mgeneqn1}), we find that
\[
\lambda_{n_1, n_2, \ldots, n_k}=\frac{(a; q)_{n_1+n_2+\cdots+n_k}}
{(c; q)_{n_1+n_2+\cdots+n_k}(q; q)_{n_1}(q; q)_{n_2}\cdots (q; q)_{n_k}}.
\]
Substituting this into (\ref{mgeneqn}), we complete the proof of
Theorem~\ref{mghahn}.
\end{proof}
\section{Another  multilinear generating function for the Hahn polynomials}
Liu \cite[Eq. (4.20)]{Liu2010} obtained the following representation of the  homogeneous
Rogers-Szeg\H{o} polynomials using the techniques of partial fraction.
\begin{prop}\label{liurspp} For any integer $k \ge 0,$  we have the identity
for the homogeneous Rogers-Szeg\H{o} polynomials,
\[
h_k(a, b|q)=\frac{a^k}{(b/a; q)_\infty}
\sum_{n=0}^\infty \frac{q^{n(k+1)}}{(q, qa/b; q)_n}
+\frac{b^k}{(a/b; q)_\infty}
\sum_{n=0}^\infty \frac{q^{n(k+1)}}{(q, qb/a; q)_n}.
\]
\end{prop}
\begin{prop}\label{liursppa}
For $\max\{|au_1|, |bu_1|, \ldots, |au_s|, |bu_s|\}<1,$ we have
\begin{align*}
&\sum_{n_1, n_2, \ldots, n_s=0}^\infty h_{n_1+n_2+\ldots+n_s+k} (a, b|q)
\frac{u_1^{n_1}u_2^{n_2}\cdots u_s^{n_s}}{(q; q)_{n_1}(q; q)_{n_2} \cdots (q; q)_{n_s}}\\
&=\frac{a^k}{(b/a, au_1, au_2, \ldots, au_s; q)_\infty}
{_{s}\phi_{s-1}} \({{au_1, au_2, \ldots, au_s}\atop{qa/b, 0, \ldots, 0}}; q,  q^{1+k}\)\\
&\quad+\frac{b^k}{(a/b, bu_1, bu_2, \ldots, bu_s; q)_\infty}
{_{s}\phi_{s-1}} \({{bu_1, bu_2, \ldots, bu_s}\atop{qb/a, 0, \ldots, 0}}; q,  q^{1+k}\).
\end{align*}
\end{prop}
\begin{proof}
If we use  $f(u_1, u_2, \ldots, u_s)$ to denote the right-hand side of the above equation, then,
using the ratio test, we find that $f$ is an analytic function of
$u_1, u_2, \ldots, u_s$ for $\max\{ |au_1|, |bu_1|, \ldots, |au_s|, |bu_s|\}<1.$
Thus, near $(0, 0, \ldots, 0)\in \mathbb{C}^s,$  $f$ can be expanded in terms of $u_1^{n_1} u_2^{n_2}\cdots u_s^{n_s},$
\[
f(u_1, u_2, \ldots, u_s)=\sum_{n_1, n_2,\ldots, n_s=0}^\infty c_{n_1, n_2, \ldots, n_s}
\frac{u_1^{n_1} u_2^{n_2}\cdots u_s^{n_s}}{(q; q)_{n_1}(q; q)_{n_2} \cdots (q; q)_{n_s}}.
\]
By a direct computation, we easily find that the coefficient $c_{n_1, n_2, \ldots, n_s}$ equals
\begin{align*}
{\partial_{q, u_1}^{n_1} \cdots \partial_{q, u_s}^{n_s}}\{f\}
\Big|_{u_1=u_2=\cdots=u_s=0}
=&\frac{a^{k+n_1+n_2+\cdots+n_s}}{(b/a; q)_\infty}
\sum_{n=0}^\infty \frac{q^{n(k+1+n_1+n_2+\cdots+n_s)}}{(q, qa/b; q)_n}\\
&\quad+\frac{b^{k+n_1+n_2+\cdots+n_s}}{(a/b; q)_\infty}
\sum_{n=0}^\infty \frac{q^{n(k+1+n_1+n_2+\cdots+n_s)}}{(q, qb/a; q)_n}.
\end{align*}
Replacing $k$ by $k+n_1+n_2+\cdots+n_s$ in Proposition \ref {liurspp}, we deduce that
\begin{align*}
h_{k+n_1+n_2+\cdots+n_s}(a, b|q)&=\frac{a^{k+n_1+n_2+\cdots+n_s}}{(b/a; q)_\infty}
\sum_{n=0}^\infty \frac{q^{n(k+1+n_1+n_2+\cdots+n_s)}}{(q, qa/b; q)_n}\\
&\quad+\frac{b^{k+n_1+n_2+\cdots+n_s}}{(a/b; q)_\infty}
\sum_{n=0}^\infty \frac{q^{n(k+1+n_1+n_2+\cdots+n_s)}}{(q, qb/a; q)_n}.
\end{align*}
Comparing the above two equations, we find that $c_{n_1, n_2, \ldots, n_s}
=h_{k+n_1+n_2+\cdots+n_s}(a, b|q).$ This completes the proof of the proposition.
\end{proof}
Srivastava and Jain noticed the following multilinear generating function
 for the Hahn polynomials \cite[p. 155]{SrivastavaJain}, and also gave a formal proof
 of the $s=2$ case of this formula using the method of $q$-integration. Next we will
 apply Theorem~\ref{principalliuhthm} to give a complete proof of this formula.

\begin{thm}\label{liusjthm}
 For $\max\{|au_1|, |bu_1|, |av_1|, |bv_1|,  \ldots, |au_s|, |bu_s|, |av_s|, |bv_s|\}<1,$ we have
\begin{align*}
&\sum_{n_1, n_2, \ldots, n_s=0}^\infty h_{n_1+n_2+\cdots+n_s+k} (a, b|q)
\frac{\Phi^{(\alpha_1)}_{n_1}(u_1, v_1|q)\Phi^{(\alpha_2)}_{n_2}(u_2, v_2|q)\cdots \Phi^{(\alpha_s)}_{n_1}(u_s, v_s|q)}
{(q; q)_{n_1}(q; q)_{n_2} \cdots (q; q)_{n_s}}\\
&=\frac{a^k(\alpha_1 au_1, \alpha_2 au_2, \ldots, \alpha_s a u_s; q)_\infty}{(b/a, au_1, av_1, \ldots, au_s, av_s; q)_\infty}
{_{2s}\phi_{2s-1}} \({{au_1, av_1, au_2, av_2, \ldots, au_s, av_s}\atop{qa/b, \alpha_1 au_1, \ldots, \alpha_s a u_s,  0, \ldots, 0}}; q,  q^{1+k}\)\\
&\quad+\frac{b^k(\alpha_1 bu_1, \alpha_2 bu_2, \ldots, \alpha_s b u_s; q)_\infty}{(a/b, bu_1, bv_1, \ldots, bu_s, bv_s; q)_\infty}
{_{2s}\phi_{2s-1}} \({{bu_1, bv_1, bu_2, bv_2,  \ldots, bu_s, bv_s}\atop{qb/a, \alpha_1 bu_1, \ldots, \alpha_s b u_s,  0, \ldots, 0}}; q,  q^{1+k}\).
\end{align*}
\end{thm}
\begin{proof}
If we use  $f(u_1, v_1, \ldots, u_s, v_s)$ to denote the right-hand side of the above equation, then,
using the ratio test, we find that $f$ is an analytic function of $u_1, v_1, u_2, v_2, \ldots, u_s, v_s$
for
\[
\max\{|au_1|, |bu_1|, |av_1|, |bv_1|,  \ldots, |au_s|, |bu_s|, |av_s|, |bv_s|\}<1.
\]
By a direct computation, we deduce that for $j=1, 2\ldots, s$,
\begin{align*}
&\partial_{q, u_j} \{f\}=\partial_{q, v_j}(1-\alpha_j \eta_{u_j})\{f\}\\
&=\frac{(1-\alpha)a^{k+1}(\alpha_1 au_1, \ldots, q\alpha_j au_j, \ldots, \alpha_s a u_s; q)_\infty}{(b/a, au_1, av_1, \ldots, au_s, av_s; q)_\infty}\\
&\quad\times
{_{2s}\phi_{2s-1}} \({{au_1, av_1, au_2, av_2, \ldots, au_s, av_s}\atop{qa/b, \alpha_1 au_1, \ldots, q\alpha_j a u_j, \ldots, \alpha_s a u_s,  0, \ldots, 0}}; q,  q^{2+k}\)\\
&\quad+\frac{(1-\alpha)b^{k+1}(\alpha_1 bu_1, \ldots, q\alpha_j bu_j, \ldots, \alpha_s b u_s; q)_\infty}{(a/b, bu_1, bv_1, \ldots, bu_s, bv_s; q)_\infty}\\
&\qquad \times {_{2s}\phi_{2s-1}} \({{bu_1, bv_1, bu_2, bv_2,  \ldots, bu_s, bv_s}
\atop{qb/a, \alpha_1 bu_1, \ldots, q\alpha_j b u_j, \ldots, \alpha_s b u_s,  0, \ldots, 0}}; q,  q^{2+k}\).
\end{align*}
Thus,  by Theorem~\ref{principalliuhthm},  there exists a sequence $\{\lambda_{n_1, n_2, \ldots, n_k}\}$ independent
of $u_1, v_1, \ldots, u_s, v_s$ such that
\begin{equation*}
f=\sum_{n_1, n_2, \ldots, n_s=0}^\infty \lambda_{n_1, n_2, \ldots, n_s}
\Phi^{(\alpha_1)}_{n_1}(u_1, v_1|q) \Phi^{(\alpha_2)}_{n_2}(u_2, v_2|q)\cdots \Phi^{(\alpha_s)}_{n_k}(u_s, v_s|q).
\end{equation*}
Setting $u_1=u_2=\cdots=u_s=0$ in the above equation and using Proposition~\ref{liursppa} , we find that
\begin{align*}
&\sum_{n_1, n_2, \ldots, n_s=0}^\infty \lambda_{n_1, n_2, \ldots, n_s}v_1^{n_1}v_2^{n_2} \cdots v_s^{n_s}\\
&=\sum_{n_1, n_2, \ldots, n_s=0}^\infty h_{n_1+n_2+\ldots+n_s+k} (a, b|q)
\frac{v_1^{n_1}v_2^{n_2}\cdots v_s^{n_s}}{(q; q)_{n_1}(q; q)_{n_2} \cdots (q; q)_{n_s}}.
\end{align*}
Equating the coefficients of $v_1^{n_1}v_2^{n_2} \cdots v_s^{n_s}$ in the both sides of the above equation, we have
\[
\lambda_{n_1, n_2, \ldots, n_s}=\frac{h_{n_1+n_2+\ldots+n_s+k} (a, b|q)}{(q; q)_{n_1}(q; q)_{n_2} \cdots (q; q)_{n_s}}.
\]
Hence, we complete the proof of Theorem~\ref{liusjthm}.
\end{proof}



\begin{thebibliography}{9}

\bibitem{SalamCarlitz} W. A. Al-Salam and  L. Carlitz,
Some orthogonal $q$-polynomials, Math. Nachr. 30
(1965), 47--61.


\bibitem{SalamVerma} W. A. Al-Salam and A. Verma,
Some remarks on $q$-beta integral, Proc. Amer. Math. Soc. 85 (1982)
360--362.



\bibitem{Andrews1972}  G. E. Andrews,
Summations and transformations for basic Appell series, J. London Math. Soc. (2) 4
(1972) 618--622.

\bibitem{Andrews+Askey} G. E. Andrews and R. Askey,
Another $q$-extension of the beta function, Proc. Amer. Math. Soc. 81 (1981) 97--100.


















\bibitem{Carlitz} L. Carlitz, Some polynomials related to theta functions,
Ann. Mat. Pure  Appl. 41 (1955) 359--373.


 \bibitem{Cao2010}J.  Cao,
  Notes on Carlitz's $q$-operators,  Taiwanese J. Math.  14 (2010) 2229--2244.

 \bibitem{Carlitz1972}L.  Carlitz,
 Generating functions for certain $q$-orthogonal polynomials, Collectanea Math. 23(1972) 91--104.
















\bibitem{Gas+Rah}G. Gasper and  M. Rahman,
Basic Hypergeometric Series, Cambridge  Univ. Press, Cambridge,
MA, 2004.



\bibitem{Hahn1} W. Hahn,
 \"{U}ber Orthogonalpolynome, die $q$-Differenzengleichungen, Math. Nuchr. 2
(1949) 4--34.
\bibitem{Hahn2} W. Hahn,
 Beitr\"{a}ge zur Theorie der Heineschen Reihen; Die $24$ Integrale der hypergeometrischen
$q$-Differenzengleichung; Das $q$-Analogon der Laplace-Transformation,
Math. Nachr. 2 (1949) 340--379.












\bibitem {Liu1997} Z.-G. Liu,
$q$-Hermite polynomials and a $q$-beta integral, Northeast. Math. J.
13 (1997) 361--366.






\bibitem{Liu2010} Z.-G. Liu,  Two $q$-difference equations and $q$-operator
identities,  J. Difference Equ. Appl. 16 (2010) 1293--1307.



\bibitem{LiuRamanujanP} Z.-G. Liu,
On the $q$-partial differential equations and $q$-series, The Legacy of Srinivasa Ramanujan, 213-250,
Ramanujan Math. Soc. Lect. Notes Ser., 20, Ramanujan Math. Soc., Mysore, 2013.





\bibitem{Malgrange} B. Malgrange,
Lectures on functions of several complex variables, Springer-Verlag, Berlin, 1984.


\bibitem{Range} R. M. Range,
Complex analysis: A brief tour into higher dimensions, Am. Math. Monthly 110
(2003) 89--108.









\bibitem{SrivastavaJain} H. M. Srivastava and V. K. Jain,
 Some multilinear generating functions for $q$-Hermite polynomials,
 J. Math Anal. Appl. 144 (1989) 147--157.





\bibitem{Taylor}J. Taylor,
Several Complex Variables with Connections to Algebraic Geometry and Lie Groups,
Graduate Studies in Mathematics, vol. 46. Am. Math. Soc., Providence,  2002.

\bibitem{Wang2009} M. Wang,
 $q$-integral representation of the Al-Salam-Carlitz polynomials,
  Appl. Math. Lett.  22  (2009) 943--945.


\end{thebibliography}
\end{document}